\documentclass[11pt]{amsart}

\usepackage[a4paper,hmargin=3.5cm,vmargin=4cm]{geometry}
\usepackage{amsfonts,amssymb,amscd,amstext}
\usepackage{graphicx}
\usepackage[dvips]{epsfig}
\usepackage{color}

\usepackage{fancyhdr}
\usepackage{verbatim}
\usepackage[utf8]{inputenc}
\usepackage{hyperref}

\input xy
\xyoption{all}

\usepackage{times}
\usepackage{enumerate}
\usepackage{titlesec}
\usepackage{mathrsfs}

\titleformat{\section}
{\filcenter\bfseries\large} {\thesection{.}}{0.2cm}{}
\titleformat{\subsection}[runin]
{\bfseries} {\thesubsection{.}}{0.15cm}{}[.]
\titleformat{\subsubsection}[runin]
{\em}{\thesubsubsection{.}}{0.15cm}{}[.]
\usepackage[up,bf]{caption}

\newtheorem{theorem}{Theorem}[section]
\newtheorem{proposition}[theorem]{Proposition}

\newtheorem{lemma}[theorem]{Lemma}
\newtheorem{corollary}[theorem]{Corollary}

\theoremstyle{definition}
\newtheorem{definition}[theorem]{Definition}
\newtheorem{remark}[theorem]{Remark}

\newtheorem{problem}[theorem]{Problem}
\newtheorem{example}[theorem]{Example}

\numberwithin{equation}{section}
\numberwithin{figure}{section}

\newcommand\C{\mathbb{C}}
\newcommand\D{\overline{\mathbb D}}

\renewcommand\D{\mathbb D}

\newcommand\R{\mathbb{R}}

\newcommand\igot{\mathfrak{i}}

\renewcommand\igot{\mathfrak{i}}

\renewcommand\imath{\igot}

\newcommand\di{\partial}

\begin{document}

\fancyhead[LO]{Curvature of minimal graphs}
\fancyhead[RE]{ D. Kalaj}
\fancyhead[RO,LE]{\thepage}

\thispagestyle{empty}

\vspace*{1cm}
\begin{center}
{\bf\LARGE  Curvature of minimal graphs}

\vspace*{0.5cm}

{\large\bf  David Kalaj}
\end{center}


\vspace*{1cm}

\begin{quote}
{\small
\noindent {\bf Abstract}\hspace*{0.1cm}
We consider the Gaussian curvature conjecture of a minimal graph $S$ over the unit disk. First of all we reduce the general conjecture to the estimating the Gaussian curvature of some Scherk's type minimal surfaces  over a quadrilateral inscribed in the unit disk containing the origin inside.
As an application, we obtain the best estimates of the Gaussian curvature so far at the point above the center of the unit disk. Further we obtain an optimal estimate of the Gaussian curvature at the point $\mathbf{w}$ over the center of the disk, provided $\mathbf{w}$ satisfies certain "symmetric" conditions. The result extends a classical result of Finn and Osserman in 1964. In order to do so, we construct a certain family $S^t$, $t\in[t_\circ, \pi/2]$ of Scherk's type minimal graphs over the isosceles trapezoids inscribed in the unit disk. Then we compare the Gaussian curvature of the graph $S$ with that of $S^t$ at the point $\mathbf{w}$ over the center of the disk.

\vspace*{0.2cm}

\noindent{\bf Keywords}\hspace*{0.1cm} conformal minimal surface, minimal graph, curvature

\vspace*{0.1cm}

\noindent{\bf MSC (2010):}\hspace*{0.1cm} 53A10, 32B15, 32E30, 32H02}

\vspace*{0.1cm}
\noindent{\bf Date: \today} 
\end{quote}

\vspace{0.2cm}


\section{Introduction}
\label{sec:intro}

Let $M\subset \R^3=\C\times \R$ be a minimal graph lying over the unit disc $\D\subset \C$.
Let $w=(w_1,w_2,w_3):\D\to M$ be a conformal harmonic parameterization of $M$ with $w(0)=0$.
Its projection $(w_1,w_2):\D\to \D$ is a harmonic diffeomorphism of the disc which may be assumed
to preserve the orientation. Let $z$ be the complex variable in $\D$, and write
$w_1+\imath w_2 = f$ in the complex notation.
We denote by $f_z=\di f/\di z$ and $f_{\bar z}=\di f/\di \bar z$ the Wirtinger derivatives of $f$.
The function $\omega$ defined by
\begin{equation}\label{eq:omega}
	\overline{f_{\bar z}} = \omega f_z
\end{equation}
is called the {\em second Beltrami coefficient} of $f$, and the above is the
{\em second Beltrami coefficient} with $f$ as a solution. Observe that $\bar{f}_z=\overline{f_{\bar z}}$ and this notation will be used in the sequel.

Orientability of $f$ is equivalent to $\mathrm{Jac}(f)=|f_z|^2-|f_{\bar z}|^2>0$, hence to
$|\omega|<1$ on $\D$. Furthermore, the function $\omega$ is holomorphic whenever
$f$ is harmonic and orientation preserving. (In general, it is meromorphic when $f$ is harmonic.)
To see this, let
\begin{equation}\label{eq:fhg}
	u+\imath v = f = h+\overline g
\end{equation}
be the canonical decomposition of the harmonic map $f:\D\to\D$,
where $h$ and $g$ are holomorphic functions on the disc. Then,
\begin{equation}\label{eq:omega2}
	f_z=h',\quad\ f_{\bar z}=\overline g_{\bar z}= \overline{g'}, \quad\
	\omega =  \overline{f_{\bar z}}/f_z = g'/h'.
\end{equation}
In particular, the second Beltrami coefficient $\omega$ equals the meromorphic function $g'/h'$
on $\D$. In our case we have $|\omega|<1$, so it is holomorphic map $\omega:\D\to\D$.

We now consider the Enneper--Weierstrass representation of the minimal graph
$\varpi=(u,v,T):\D \to M\subset \D\times \R$ over $f$, following Duren \cite[p.\ 183]{Duren2004}. We have
\begin{eqnarray*}
	u(z) &=& \Re f(z) = \Re \int_0^z \phi_1(\zeta)d\zeta \\
	v(z) &=& \Im f(z) = \Re \int_0^z \phi_2(\zeta)d\zeta \\
	T(z) &=& \Re \int_0^z \phi_3(\zeta)d\zeta
\end{eqnarray*}
where
\begin{eqnarray*}
	\phi_1 &=& 2(u)_z = 2(\Re f)_z = (h+\bar g + \bar h + g)_z = h'+g', \\
	\phi_2 &=& 2(v)_z = 2(\Im f)_z = \imath(\bar h+g - h -\bar g)_z = \imath(g'-h'), \\
	\phi_3 &=& 2(T)_z = \sqrt{-\phi_1^2-\phi_2^2} = \pm 2\imath \sqrt{h'g'}.
\end{eqnarray*}
The last equation follows from the identity $\phi_1^2+\phi_2^2+\phi_3^2=0$ which
is satisfied by the Enneper--Weierstrass datum $\phi=(\phi_1,\phi_2,\phi_3)=2\di w$
of any conformal minimal (equivalently, conformal harmonic) immersion $w:D\to\R^3$
from a conformal surface $D$. Let us introduce the notation $p=f_z$. We have that
\begin{equation}\label{eq:p}
	p = f_z = (\Re f)_z + \imath (\Im f)_z = \frac12(h'+g' + h'-g') = h'.
\end{equation}
By using also $\omega =  \overline{f_{\bar z}}/f_z = g'/h'$ (see \eqref{eq:omega2}), it follows that
\[
	\phi_1 = h'+g'=p(1+\omega),\quad \phi_2 = -\imath(h'-g')=-\imath p(1-\omega),\quad
	\phi_3 = \pm 2\imath p \sqrt{\omega}.
\]
From the formula for $\phi_3$ we infer that $\omega$ has a well-defined holomorphic square root:
\begin{equation}\label{eq:q}
	\omega = q^2,\qquad q:\D\to \D\ \ \text{holomorphic}.
\end{equation}
In terms of the Enepper--Weierstrass parameters $(p,q)$ given by \eqref{eq:p} and \eqref{eq:q} we obtain
\begin{equation}\label{eq:EW}
	\phi_1 = p(1+q^2),\quad \phi_2 = -\imath p(1-q^2),\quad
	\phi_3 = -2\imath p q.
\end{equation}
(The choice of sign in $\phi_3$ is a matter of convenience; since we have two choices of sign for
$q$ in \eqref{eq:q}, this does not cause any loss of generality.) Hence,
\[
	\varpi(z) = \left(\Re f(z), \Im f(z), \Im \int_0^z 2 p(\varsigma) q(\varsigma) dt \right),\quad z\in\D.
\]

The curvature $\mathcal{K}$ of the minimal graph $M$ is expressed in terms of $(h,g,\omega)$ \eqref{eq:omega2}, and in terms
of the Enneper--Weierstrass parameters $(p,q)$, by
\begin{equation}\label{eq:curvatureformula}
	\mathcal{K} = - \frac{|\omega'|^2}{|h'g'|(1 + |\omega|)^4} = - \frac{4|q'|^2}{|p|^2(1 + |q|^2)^4},
\end{equation}
where $p=f_z$ and $\omega=q^2=\overline{f_{\bar z}}/f_z$. (See Duren \cite[p.\ 184]{Duren2004}.)

\subsection{Non-parametric minimal surface equation}

Assume that $S=\{(u,v, \mathbf{f}(u,v)):(u,v)\in\D\}$ is a minimal surface, where $\D$ is the unit disk. Then we call such a surface minimal surface above the unit disk. The minimal surface equation is $$f_{uu}(1+f_v^2)-2 f_u f_v f_{uv}+f_{vv}(1+f_u^2)=0.$$

\section{The Heinz-Hopf-Finn-Osserman problem}
We are interested in the following problem.

%
%
\begin{problem}\label{problem}
What is the supremum of $|\mathcal{K}(\mathbf{w})|$ over all minimal graphs lying over $\D$? Is
\begin{equation}\label{eq:FinnOsserman}
	|\mathcal{K}(\mathbf{w})|< \frac{\pi^2}{2}
\end{equation}
the precise upper bound? Here $\mathbf{w}$ is the point above the center of the unit disk and we call it \emph{centre}.
\end{problem}
The previous conjecture has been also formulated by Duren in his monograph \cite[Conjecture~2.~p.~185]{Duren2004}.

The first result on this topic has been given by E. Heinz on 1952 in \cite{zbMATH03075392} who introduced the constant $c_0$ which is the best constant in the inequality $|\mathcal{K}(\mathbf{w})|\le c_0$, for all minimal graphs over the unit disk with the centre $\mathbf{w}$. Further this result has been improved by E. Hopf in 1953 in \cite{zbMATH03081064}, who introduced the constant $c_1$ which is the best constant in the inequality $$W^2|\mathcal{K}(\mathbf{w})|\le c_1, $$ where $W=\sqrt{1+\mathbf{f}_u^2+\mathbf{f}_v^2}$. So a similar problem to be consider is the following
\begin{problem}\label{problem2}
What is the supremum of $W^2|\mathcal{K}(\mathbf{w})|$ over all minimal graphs lying over $\D$? Is
\begin{equation}\label{eq:FinnOsserman3}
	W^2|\mathcal{K}(\mathbf{w})|< \frac{\pi^2}{2}
\end{equation}
the precise upper bound? Here $\mathbf{w}$ is the \emph{centre} of minimal surface.
\end{problem}

It was shown by Finn and Osserman \cite{FinnOsserman1964} in 1964 that the
upper bound in \eqref{eq:FinnOsserman} is indeed sharp if  $q(0)=0$, which means that
the tangent plane $T_0 M=\C\times\{0\}$ being horizontal (and hence $f$ is conformal at $0$).
Although there is no minimal graph lying over the whole unit disc $\D$ whose centre curvature equals
$\frac{\pi^2}{2}$, there is a sequence of minimal graphs whose centre curvatures converge to  $\frac{\pi^2}{2}$,
and the graphs converge to the Scherk's surface lying over square inscribed into the unit disc.
The associated Beltrami coefficient of the Scherk's surface is $\omega(z)=z^2$, with $q(z)=z$.
We refer to Duren \cite[p.\ 185]{Duren2004} for a survey of this subject. We also refer to the monograph by J. C. C. Nitsche \cite{Nitsche1965} for earlier results.

Let us recall a path to obtain a weaker upper bound on $|\mathcal{K}|$ which holds for every value
$|q(0)|<1$. This is explained in \cite[pp.\ 184--185]{Duren2004}.

Hall proved in \cite{Hall1982} (1982) the following estimate
\begin{equation}\label{eq:Hall}
	|f_z(0)|^2+|f_{\bar z}(0)|^2\ge \frac{27}{4\pi^2}
\end{equation}
for any harmonic diffeomorphism $f:\D\to \D$ with $f(0)=0$.
This estimate is sharp in general, but is not sharp if the second Beltrami coefficient
$\omega$ is the square of a holomorphic function on $\D$.
Applying Hall's estimate and noting that
\[
	|f_z(0)|^2+|f_{\bar z}(0)|^2 = |f_z(0)|^2(1+|q(0)|^4)
\]
gives
\[
	|f_z(0)|^2 \ge \frac{27}{4\pi^2} \frac{1}{(1+|q(0)|^4)}.
\]
By using also the Pick-Schwarz inequality $|q'(0)|<1-|q(0)|^2$, we obtain
\begin{equation}\label{eq:weakestK}
	|\mathcal{K}| = \frac{4|q'(0)|^2}{|f_z(0)|^2(1+|q(0)|^2)^4} \le
	\frac{16\pi^2}{27} \frac{\bigl(1-|q(0)|^2\bigr)^2 \bigl(1+|q(0)|^4\bigr)}{(1+|q(0)|^2)^4}.
\end{equation}
So we have the following inequality \begin{equation}\label{halljda}
	|\mathcal{K}| \le \frac{16\pi^2}{27}\approx 5.84865.
\end{equation}
The above constant is better than the constant $5.98$ obtained by Finn and Osserman in \cite{FinnOsserman1964}.

Further if the minimal surface has its non-parametric parameterization  $z=\mathbf{f}(u,v)$, and denoting $$W=\sqrt{1+\mathbf{f}_u^2+\mathbf{f}_v^2},$$ then
\eqref{eq:weakestK}, in view of \eqref{firsti} and \eqref{secondi} below implies that
\begin{equation}\label{eq:weakestK2}
	|\mathcal{K}|\cdot W^2   \le
	\frac{16\pi^2}{27} \frac{ \bigl(1+|q(0)|^4\bigr)}{(1+|q(0)|^2)^2}\le \frac{16\pi^2}{27}.
\end{equation}
It follows from \eqref{eq:weakestK}, that the Heinz constant $c_0<\frac{16\pi^2}{27}$, while \eqref{eq:weakestK2} implies that the Hopf constant $c_1<\frac{16\pi^2}{27}$.

We will give better estimate of both constants in Corollary~\ref{coro}.

The estimate \eqref{halljda} is not sharp as it has been proved by R. Hall in \cite{Hall1998} by obtaining a very small improvement of about $10^{-5}$.
As said before, the sharp estimate \eqref{eq:FinnOsserman} in the case $q(0)=0$ was given by
Finn and Osserman \cite{FinnOsserman1964} (see also \cite{Nitsche1965}).

\section{The main results}
We first formulate the following general  result
\begin{theorem}\label{prejprej}
For every $w\in\D$, there exist four different points $a_0, a_1,a_2,a_3\in\mathbf{T}$  and then there is a harmonic mapping $f$ of the unit disk onto the quadrilateral $Q(a_0,a_1,a_2,a_3)$ that solves the Beltrami equation \begin{equation}\label{beleq}\bar f_z(z) = \left(\frac{w+\frac{\imath \left(1-w^4\right) z}{\left|1-w^4\right|}}{1+\frac{\imath\overline{w} \left(1-w^4\right) z}{\left|1-w^4\right|}}\right)^2 f_z(z),\end{equation} $|z|<1$ and satisfies the initial condition $f(0)=0$, $f_z(0)>0$. It also defines a Scherk's type minimal surface $S^\diamond: \zeta=\mathbf{f}^\diamond(u,v)$ over the  quadrilateral $Q(a_0,a_1,a_2,a_3)$, with the centre $\mathbf{w}=(0,0,0)$ so that its Gaussian normal is $$\mathbf{n}^\diamond_{\mathbf{w}}=-\frac{1}{1+|w|^2}(2\Im w, 2\Re w, -1+|w|^2),$$ and $D_{uv}\mathbf{f}^\diamond(0,0)=0$. Moreover, every other non-parametric minimal surface $S:$ $z=\mathbf{f}(u,v)$ over the unit disk, with a centre  $\mathbf{w}$, with $\mathbf{n}_{\mathbf{w}}=\mathbf{n}^\diamond_{\mathbf{w}}$ and $D_{uv}\mathbf{f}(0,0)=0$ satisfies the sharp inequality $$|\mathcal{K}_{S}(\mathbf{w})|<|\mathcal{K}_{S^\diamond}(\mathbf{w})|,$$ or what is the same
$$W^2_{S}|\mathcal{K}_{S}(\mathbf{w})|<W^2_{S^\diamond}|\mathcal{K}_{S^\diamond}(\mathbf{w})|.$$
Further we have \begin{equation}\label{finoser}\mathcal{K}_{S^\diamond}(\mathbf{w})=-\frac{4 \left(1-|w|^2\right)^2}{\left(1+|w|^2\right)^4 |f_z(0)|^2}.\end{equation}
\end{theorem}
\begin{remark}
 It follows from the result of Jenkins and Serrin that such a minimal surface described in Theorem~\ref{prejprej} is unique \cite{Jenkins1Serrin1968}, so $Q=Q(w)$ depends only on $w$ and also $f=f^w$ depends only on $w$. It also follows from Theorem~\ref{prejprej} (i.e. from \eqref{finoser}) that the Heinz and the Hopf constants can be defined as
 \begin{equation} \label{heinz}c_0 = \sup_{w} \frac{4 \left(1-|w|^2\right)^2}{\left(1+|w|^2\right)^4 |f^w_z(0)|^2}\end{equation}

 \begin{equation} \label{heinz}c_1 = \sup_{w} \frac{4 }{\left(1+|w|^2\right)^2 |f^w_z(0)|^2}.\end{equation}

In particular, when $w$ from  Theorem~\ref{prejprej} is an imaginary number (in view of Remark~\ref{remica}), then we precisely describe the quadrilaterals, which appear to be isosceles trapezoids (Section~\ref{sectio2}, Proposition~\ref{defshre}). In this case we give the precise bound of the curvature.

Further, if we consider the mapping $$\tilde f(z) = f\left(\frac{\imath\left|1-w^4\right|}{1-w^4}\frac{ (w-z) }{ (1-z \overline{w})}\right),$$ then $\tilde f$ satisfies the Beltrami equation $\overline{\tilde f}_z=z^2 \tilde f_z$ with the initial conditions $\tilde f(w)=0$ and $\imath (1-w^4)\tilde f_z(w)>0$.
\end{remark}

In order to formulate our next results, which are extensions of the Finn-Osserman results, we give the following definition.
\begin{definition}
We call  $\zeta\in D$ a symmetric point of a double differentiable real function $\mathbf{f}:D\to \mathbf{R}$ if  there is some vector $h\in\mathbf{T}=\partial\D$ so that the equalities hold \begin{equation}\label{symm}\nabla^2_{h, \imath h}\mathbf{f}(\zeta) = \nabla _h \mathbf{f}(\zeta)=0.\end{equation} We call also that point $\zeta$, $h-$symmetric. A point $\mathbf{w}=(\zeta, \mathbf{f}(\zeta))$ on the graph of a function $\mathbf{f}$ is symmetric if $\zeta$  is symmetric for $\mathbf{f}$.

\end{definition}
\begin{remark}
The motivation for this definition comes from the following observation. Assume that $\mathbf{f}$ is a symmetric real function w.r.t. imaginary axis, i.e. assume that $\mathbf{f}(-u,v)=\mathbf{f}(u,v)$. Then $D_u \mathbf{f}(-u,v)=-D_u \mathbf{f}(u,v)$. So $D_u\mathbf{f}(0,v)=0$. Further $D_{uv} \mathbf{f}(0,v)=0$ for every $v$. This implies that $\nabla^2_{e_1,e_2} \mathbf{f}(0,0)=0$. By using the translation and rotation of the coordinate system, we get a similar fact for functions that are symmetric at some point w.r.t to an arbitrary line, or more general w.r.t. a small segment.

\end{remark}
An example of a symmetric point is any stationary point of the function.
\begin{example}\label{forward}
Prove that if  $\nabla \mathbf{f}(0,0)=0$, then $z=(0,0)$ is a symmetric point of $\mathbf{f}$. Namely if $h=e^{ic}$ and $\mathbf{f}^c(z) = \mathbf{f}(e^{ic}z)$, then
$$\mathbf{f}^c_u(0,0)=\cos c \,\mathbf{f}_u(0,0)+\sin c \,\mathbf{f}_v(0,0)=\nabla_h \mathbf{f}(0,0).$$ Further
$$\mathbf{f}^c_{uv}(0,0)=\cos(2 c) \mathbf{f}_{uv}(0,0)+\cos (c) \sin(c)  \left(-\mathbf{f}_{vv}(0,0)+\mathbf{f}_{uu}(0,0)\right)=\nabla^2_{h,\imath h} \mathbf{f}(0,0).$$
Since $\mathbf{f}^{\pi/2}_{uv}(0,0)=-\mathbf{f}_{uv}(0,0)$, there is $c$ so that $\mathbf{f}^c_{uv}(0,0)=0$.
\end{example}
In the sequel we give two additional examples of symmetric points of  classical minimal surfaces and one counterexample.
\begin{example}
a) Assume that $w=\cosh^{-1}\sqrt{u^2+v^2}$, $|w|=\sqrt{u^2+v^2}>1$. Then this function defines the catenoid. Moreover $$w_u=\frac{u}{\sqrt{u^2+v^2} \sqrt{-1+\sqrt{u^2+v^2}} \sqrt{1+\sqrt{u^2+v^2}}}$$ and $$w_{uv}=\frac{u v \left(1-2 u^2-2 v^2\right)}{\left(u^2+v^2\right)^{3/2} \left(-1+\sqrt{u^2+v^2}\right)^{3/2} \left(1+\sqrt{u^2+v^2}\right)^{3/2}}.$$ So every point $(u,0)$ and $(0,v)$ is a symmetric point of this surface. Since it is rotation invariant, it follows that every point of this surface is symmetric.

b) Assume that $w=\log\frac{\cos v}{\cos u}$. Then $w_u=\tan u$ and $w_{uv}=0$. So every point $w=\imath v=(0,v)$ is a symmetric point of Scherk's saddle surface.

c) Assume that $w=\tan^{-1}\frac{v}{u}$, where $u\neq 0$. Then this function defines the helicoid.  Then $w_u=\frac{v}{u^2+v^2}$ and $w_{uv}=\frac{(u-v) (u+v)}{\left(u^2+v^2\right)^2} $. It follows that this surface has not any symmetric point.
\end{example}
We give a partial solution of  Problem~\ref{problem} and extend Finn-Osserman result by proving the following theorem.
\begin{theorem}\label{th:theor}
Assume that $S$ is a non-parametric minimal surface above the unit disk and assume that the point $\mathbf{w}$ over the center of the disk is symmetric. Then the Gaussian curvature  $\mathcal{K}(\mathbf{w})$  satisfies the sharp inequalities \begin{equation}\label{eq:FinnOsserman1}
	|\mathcal{K}(\mathbf{w})|< \frac{\pi^2}{2}.
\end{equation} and
\begin{equation}\label{eq:FinnOsserman2}
	W^2|\mathcal{K}(\mathbf{w})|< \frac{\pi^2}{2}.
\end{equation}
\end{theorem}
\begin{remark}
After we wrote this paper we realized that the statement of Theorem~\ref{th:theor} is not new for symmetric minimal surfaces. An  approach different from our approach has been given by Nitsche in \cite{zbMATH03431423}.
\end{remark}
Further we prove the following theorem
\begin{theorem}\label{th:theor2}
There is a decreasing diffeomorphism  $\Psi:[0,\pi/2]\to [0,\pi^2/2]$ with the following property.
Assume that $S$ is a non-parametric minimal surface above the unit disk with a $h-$symmetric point $\mathbf{w}$ above $0$.  Assume that $\theta$ is the angle of the tangent plane $TS_\mathbf{w}$ at $\mathbf{w}$ with $h$. Then the Gaussian curvature  $|\mathcal{K}|$ at $\mathbf{w}$  satisfies the sharp inequality \begin{equation}\label{eq:FinnOsserman2}
	|\mathcal{K}(\mathbf{w})|< \Psi(\theta)(<\frac{\pi^2}{2}),
\end{equation}
and for every $0\le \phi<\Psi(\theta)$ there is a non-parametric minimal surface $S_\phi$ above the unit disk, whose point above the center of the unit disk is $h-$simmetric and whose  tangent plane at $\mathbf{w}$ makes the angle $\theta$ with $h$ so that $$\mathcal{K}_{S_\phi}(\mathbf{w})=\phi.$$
\end{theorem}

As a corollary of our results, we obtain the following improvement of the Hall upper bound of Gaussian curvature (i.e. of Heinz and Hopf constants) without any condition on the centre.
\begin{corollary}\label{coro}
Let $S: z=\mathbf{f}(u,v)$ be a minimal surface over the unit disk and assume that $\mathbf{w}$ is its centre. Then the Gaussian curvature $$\mathcal{K}(\mathbf{w})< 5.7.$$ Moreover if $W=\sqrt{1+\mathbf{f}_u^2+\mathbf{f}_v^2}$, then   $$\mathcal{K}(\mathbf{w})< \frac{5.8}{W^2}.$$
\end{corollary}
\section{Proof of main results}
This section contains the proof of our results. At the begging we describe the family of Scherk's type minimal surfaces over isosceles trapezoids inscribed in the unit disk. On account of Theorem~\ref{prejprej} we know that a similar family depending on two parameters exists and such a family would solve the general conjecture, provided it can be explicitly expressed.

\subsection{Scherk's type minimal surfaces with 4 sides and auxiliary results}\label{sectio2}
We are going to find a harmonic mapping of the unit disk onto a quadrilateral
 inscribed in the unit disk that produces a minimal surface. Let $a_1 =1$, $a_2=e^{\imath t}$, $a_3=e^{\imath s}$, $a_4=e^{\imath(t+s)}$, where $s=\arccos \frac{3\cos t-1}{1+\cos t}.$
Let
$$F(\sigma)=\begin{array}{ll}
 \Bigg\{ &
\begin{array}{ll}
 1 & \sigma\in[0,\pi/2] \\
 e^{\imath t} & \sigma\in[\pi/2,\pi]  \\
 e^{\imath s} & \sigma\in [\pi, 3\pi/2] \\
 e^{\imath (t+s)} & \sigma\in [3\pi/2,2\pi].
\end{array}
\end{array}$$
Let $$f_1(z) =P[F](z) = \frac{1}{2\pi}\int_0^{2\pi}\frac{1-r^2}{1+r^2-2 r \cos(\varsigma-\sigma)} F(\sigma)d \sigma, z=re^{\imath\varsigma}.$$

Then $f_1$ maps the unit disk onto the trapezoid $\mathcal{T}$ with the vertices $a_1,a_2, a_3, a_4$, $a_k=a_1$. Moreover $$f_1(0) = \frac{1}{4} \left(1+e^{\imath t}\right) \left(1+e^{\imath \cos^{-1}\left[\frac{-1+3 \cos t}{1+\cos t}\right]}\right).$$
Further, by \cite[p.~63]{Duren2004},  $$f_1(z) = g(z) +\overline{h(z)},$$ where
$$g'(z) =\frac{1}{2\pi \imath} \sum_{k=1}^4 \frac{(a_k-a_{k+1})}{z-\imath ^k} $$ and
$$h'(z) =-\frac{1}{2\pi \imath} \sum_{k=1}^4 \frac{(\overline{a_k-a_{k+1}})}{z-\imath ^k}.$$ Thus
$$g'(z) =-\frac{(1+\imath) \left(\imath+e^{\imath t}\right) \left(-1+\cos t+2 \imath \sqrt{\cos t} \sin \left[\frac{t}{2}\right]\right) \left(1+\frac{z \sqrt{\cos t}}{\cos \left[\frac{t}{2}\right]+\sin \left[\frac{t}{2}\right]}\right)^2}{\pi \left(z^4-1\right) (1+\cos t)},$$
and
$$h'(z) = \frac{(1+\imath) \left(1-\cos t+2 \imath \sqrt{\cos t} \sin \left[\frac{t}{2}\right]\right) (e^{-\imath t} \left(\imath+e^{\imath t}\right))\left(z+\frac{\sqrt{\cos t}}{\cos \left[\frac{t}{2}\right]+\sin \left[\frac{t}{2}\right]}\right)^2 }{\pi  \left(z^4-1\right) (1+\cos t)}.$$
Thus we get  $$\omega_1=\frac{h'(z)}{g'(z)}=e^{-\imath \left(t+s-\pi\right)}\frac{ \left(z+\frac{\sqrt{\cos t}}{\cos \left[\frac{t}{2}\right]+\sin \left[\frac{t}{2}\right]}\right)^2}{\left(1+\frac{z \sqrt{\cos t}}{\cos \left[\frac{t}{2}\right]+\sin \left[\frac{t}{2}\right]}\right)^2},$$ where $s=\cos^{-1}\frac{3\cos t - 1}{1+\cos t}$.
So $\omega_1=q_1^2$, where  $$q_1(z)=e^{\imath\mu}\frac{z+a(t)}{1+z\overline{a(t)}}.$$
Here $$a(t) = \frac{\sqrt{\cos t}}{\cos \left[\frac{t}{2}\right]+\sin \left[\frac{t}{2}\right]},$$ and $$\mu = -\frac{1}{2}\left(t+s-\pi \right).$$

Let $\tau = \frac{1}{2} \left(-\pi +t+s\right)$ and define  \begin{equation}\label{after}f: \D\to \mathcal{T}, \ \ f(z):=e^{-\imath\tau} f_1(z). \end{equation} Then $f$ maps the unit disk onto the isosceles trapezoid, whose bases are parallel to the real axis. See figure~3.1.
\begin{figure}[htp]\label{f1}
\centering
\includegraphics{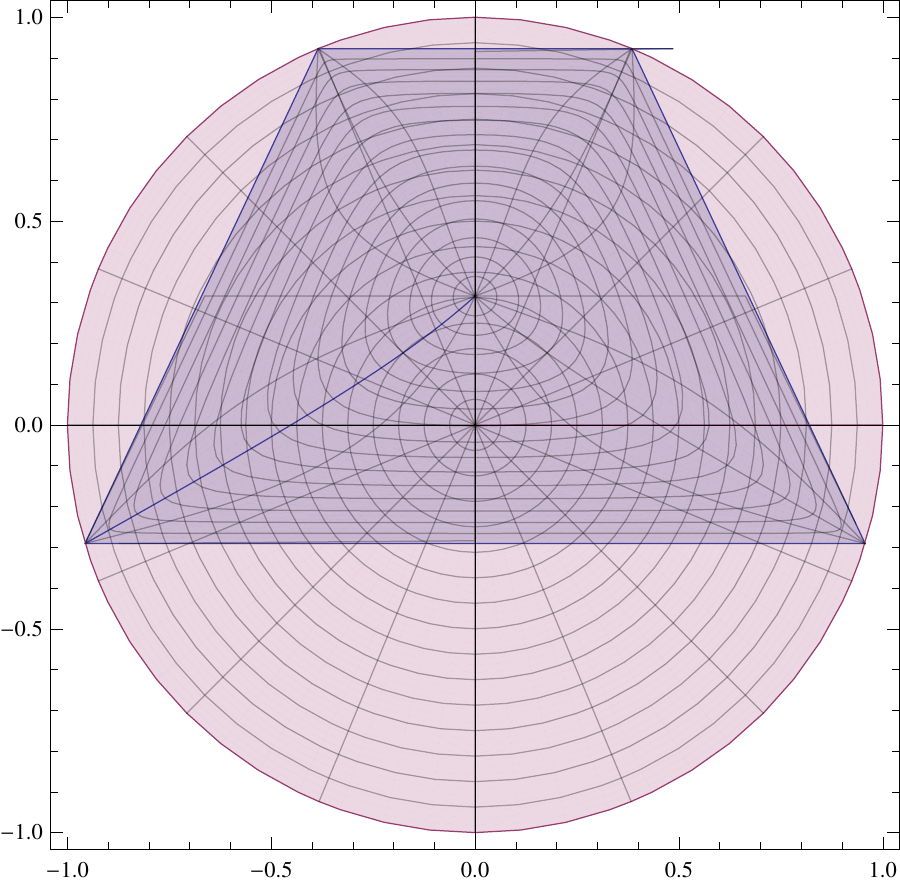}
\caption{An isosceles trapezoid inscribed in the unit disk. Here $t=\pi/2-0.1$}
\end{figure}
Then
\begin{equation}\label{be}f(0) = e^{-\imath\tau}\frac{1}{4}(1+e^{\imath t}+e^{\imath s}+e^{\imath (t+s)})=\imath \sqrt{\cos t}.\end{equation}
Further let
$$p=e^{-\imath\tau} g'(z), \tilde p = e^{\imath\tau} h'(z).$$

Then \begin{equation}\label{afterv}p= \frac{-2 \imath \left(\cos \left[\frac{t}{2}\right]+\sin\left[\frac{t}{2}\right]\right) \sin t
\left(1+\frac{z \sqrt{\cos t}}{\cos \left[\frac{t}{2}\right]+\sin \left[\frac{t}{2}\right]}\right)^2}{\pi \left(1-z^4\right) (1+\cos t)},\end{equation}
and

$$\tilde p= \frac{-2 \imath \left(\cos \left[\frac{t}{2}\right]+\sin\left[\frac{t}{2}\right]\right) \sin t \left(z+\frac{ \sqrt{\cos t}}{\cos \left[\frac{t}{2}\right]+\sin \left[\frac{t}{2}\right]}\right)^2}{\pi \left(1-z^4\right) (1+\cos t)},$$
 and $$\frac{\tilde p}{p}=\frac{ \left(z+\frac{\sqrt{\cos t}}{\cos \left[\frac{t}{2}\right]+\sin \left[\frac{t}{2}\right]}\right)^2}{\left(1+\frac{z \sqrt{\cos t}}{\cos \left[\frac{t}{2}\right]+\sin \left[\frac{t}{2}\right]}\right)^2}.$$

 Thus
\begin{equation}\label{qqq}q=\frac{z+\frac{\sqrt{\cos t}}{\cos \left[\frac{t}{2}\right]+\sin \left[\frac{t}{2}\right]}}{1+\frac{z \sqrt{\cos t}}{\cos \left[\frac{t}{2}\right]+\sin \left[\frac{t}{2}\right]}},\end{equation}

\begin{equation}\label{pmtp}p-\tilde p = -\frac{8 \imath \csc t \sin \left[\frac{t}{2}\right]^3}{\pi  \left(1+z^2\right)}\end{equation}

and

\begin{equation}\label{pptp}p+\tilde p = \frac{4 \imath \left(\left(1+z^2\right) \cos \left[\frac{t}{2}\right] \sin \left[\frac{s}{2}\right]+2 z \cos \left[\frac{s}{2}\right] \sin \left[\frac{t}{2}\right]\right)}{\pi(z^4-1)}.\end{equation}
We also have $ f(z) = u + \imath v+f(0)$, where
$$u=\Re\int_0^z (p+\tilde p)dz,$$ and
$$v=\Im \int_0^z  (p-\tilde p)dz.$$
Thus we obtain that $$u(z) = -\frac{\Im\left[\log\left[\frac{1+z^2}{(1+z)^2}\right] \sin \left[\frac{s-t}{2}\right]+\log\left[\frac{(1-z)^2}{1+z^2}\right] \sin \left[\frac{s+t}{2}\right]\right]}{\pi },$$ and
$$v(z) = -\frac{4 \Re(\tan^{-1} z) \sin \left[\frac{s}{2}\right] \sin \left[\frac{t}{2}\right]}{\pi}.$$
So the equation $$ v(r)+\sqrt{\cos t}=0$$ has only one solution $$r = \tan \left[\frac{1}{8} \pi  \sqrt{\cos t} \csc\left[\frac{t}{2}\right]^3 \sin t\right].$$
So $ f(z_\circ)=0$ if \begin{equation}\label{zezero}z_\circ = \tan \left[\frac{1}{8} \pi  \sqrt{\cos t} \csc\left[\frac{t}{2}\right]^3 \sin t\right].\end{equation}

The Gaussian curvature of the minimal surface at the point $\mathbf{w}$ over the point $0=f(z_\circ)$ is $$K= -\frac{4|q'(z_\circ)|^2}{|p(z_\circ)|^2(1+|q(z_\circ)|^2)^4}$$ which can be written as $K=-\kappa^2(t)$, where
$\kappa$ is a positive function defined by  $$\kappa(t) = \frac{2|q'(z_\circ)|}{|p(z_\circ)|(1+|q(z_\circ)|^2)^2}.$$

\subsubsection{Show that $\kappa$ is increasing  and $\kappa(t) \le \kappa(\pi/2)=\frac{\pi}{\sqrt{2}}$}\label{subsub1}
For $r=z_\circ\in(0,1)$, by direct computations we get  \begin{equation}\label{kappa}\kappa(t) = \frac{\pi  \left(1-r^4\right) \cos \left[\frac{t}{2}\right]}{2 \left(\left(1+r^2\right) \cos \left[\frac{t}{2}\right]+2 r \sqrt{\cos t}\right)^2}.\end{equation}
Notice that for $t_\circ =2 \tan^{-1}\sqrt{1/2 (-1 + \sqrt{5})}\approx 1.33248$ and $t\in(t_\circ,\pi/2)$, $z_\circ=z_\circ(t)\in\D$. For $t=t_\circ$,  $z_\circ=1$ and for $t<t_\circ$, $z_\circ$ is outside of the unit disk. Let's choose the substitution $u = \tan\frac{t}{2}$, then $u\in[\sqrt{1/2 (-1 + \sqrt{5})},1]$ and
$$\kappa =\Phi(u):= \frac{\pi  \sqrt{1+u^2} \cos \left[\frac{\pi  \sqrt{1-u^2}}{2 u^2}\right]}{2 \left(1+\sqrt{1-u^2} \sin \left[\frac{\pi  \sqrt{1-u^2}}{2 u^2}\right]\right)^2}.$$ Further we have
\[\begin{split}\Phi'(u) &= \pi  \Bigg(3 \pi  \sqrt{1-u^2} \left(2+u^2-u^4\right)+4 u^4 \sqrt{1-u^2} \cos \left[\frac{\pi  \sqrt{1-u^2}}{2 u^2}\right]\\&
+\pi(2+u^2-u^4)\left(\sqrt{1-u^2}\cos \left[\frac{\pi  \sqrt{1-u^2}}{u^2}\right]+2 \sin \left[\frac{\pi  \sqrt{1-u^2}}{2u^2}\right]\right)\Bigg)\\&\Bigg/\left(8 \sqrt{1-u^4} \left(u+u \sqrt{1-u^2} \sin \left[\frac{\pi  \sqrt{1-u^2}}{2 u^2}\right]\right)^3\right).\end{split}\]
In order to show that $\Phi'(u)>0$ we only need to prove that $$\gamma(u):= \pi(2+u^2-u^4)\left(\sqrt{1-u^2}\cos \left[\frac{\pi  \sqrt{1-u^2}}{u^2}\right]+2 \sin \left[\frac{\pi  \sqrt{1-u^2}}{2u^2}\right]\right)\ge 0$$ because the other terms in the sum are positive due to the fact that
$\frac{\pi  \sqrt{1-u^2}}{2u^2} \in(0,\pi/2)$.

Further we have
\[\begin{split}\gamma(u) &\ge  \pi(2+u^2-u^4)\left(-\sqrt{1-u^2}+2 \sin \left[\frac{\pi  \sqrt{1-u^2}}{2u^2}\right]\right)\\&\ge  \pi(2+u^2-u^4)\left(-\sqrt{1-u^2}+2\frac{2}{\pi}\left[\frac{\pi  \sqrt{1-u^2}}{2u^2}\right]\right)
\\&=\pi(2+u^2-u^4)\left(-\sqrt{1-u^2}+\frac{2 \sqrt{1-u^2}}{u^2}\right)>0.\end{split}\]
This implies that $\kappa$ is an increasing function for $t\in[t_\circ, \pi/2]$ so that $\kappa(t_\circ)=0\le \kappa(t)\le \kappa(\pi/2)=\pi/\sqrt{2}$.

Let  $W=\frac{1+|q|^2}{1-|q|^2}.$ Then define  $$\phi(t):=W \kappa(t) = \frac{\pi  \left(1+z^2\right) \cot \left[\frac{t}{2}\right]}{2 \left(1+z^2\right) \cos \left[\frac{t}{2}\right]+4 z \sqrt{\cos t}}.$$
\subsubsection{Show that $\phi(t)\le \pi/\sqrt{2}$ for $t\in(t_\circ, \pi/2)$}\label{subsub2}
 Straightforward calculations give $$\phi(t)=\frac{\pi  \cot \left[\frac{t}{2}\right]}{2 \cos \left[\frac{t}{2}\right]+2 \sqrt{\cos t} \sin \left[\frac{1}{4} \pi  \sqrt{\cos t} \csc \left[\frac{t}{2}\right]^3 \sin t\right]}$$
 or what is the same
 $$\phi(t)=\frac{\pi  }{2 \sin \left[\frac{t}{2}\right]+2 \sqrt{\cos t} \sin \left[\frac{1}{4} \pi  \sqrt{\cos t} \csc \left[\frac{t}{2}\right]^3 \sin t\right] \tan \left[\frac{t}{2}\right]},$$ and we need to show that
 $$\psi(t):={ \sin \left[\frac{t}{2}\right]+ \sqrt{\cos t} \sin \left[\frac{1}{4} \pi  \sqrt{\cos t} \csc \left[\frac{t}{2}\right]^3 \sin t\right] \tan \left[\frac{t}{2}\right]}\ge\frac{ \sqrt{2}}{2}.$$
 Since  $t\in(t_\circ, \pi/2)$, we have  $$u=  \frac{1}{4} \pi  \sqrt{\cos t} \csc \left[\frac{t}{2}\right]^3 \sin t\in[0,\pi/2]$$ and thus $$\sin u\ge \frac{2}{\pi}u.$$
 So $$\psi(t)\ge \sin \left[\frac{t}{2}\right]+\frac{1}{2} \cos t \csc \left[\frac{t}{2}\right]^2 \sec\left[\frac{t}{2}\right] \sin t$$ or what is the same \begin{equation}\label{rhs}\psi(t)\ge \vartheta(t):=\cos \left[\frac{t}{2}\right] \cot \left[\frac{t}{2}\right].\end{equation} Now the derivative of $\vartheta(t)$ is $\frac{1}{4} (-3+\cos t) \cot \left[\frac{t}{2}\right] \csc \left[\frac{t}{2}\right]$, so $\vartheta(t)$ is decreasing. Thus  $\psi(t)\ge \vartheta(\pi/2)=\sqrt{2}/2$ for $t\in(t_\circ, \pi/2)$. This implies the claimed inequality.

Observe that $s>t$ and for $t\in(t_\circ,\pi/2]$  the (isosceles) trapezoid
 $R$ contains $0$. For $t=t_\circ$,  $\mathcal{T}$ is a certain isosceles trapezoid with the base consisted of the diameter $[-1,1]$.

Let $$S^t =\{(\Re f(z), \Im f(z), T(z)): z\in\D\}.$$ Then $S^t$ is a Scherk's type minimal graph.

The third coordinate  of the Enneper-Weierstrass parametrization is given by $$T(z) = \pm \Re \int_0^z \sqrt{p\tilde p}dz.$$
So $$ T(z)=\pm 2\Re\int_0^z\frac{2 \left(z \sqrt{2-2 \cos t}+\left(1+z^2\right) \sqrt{\cos t} \tan \left[\frac{t}{2}\right]\right)}{\pi  \left(-1+z^4\right)}dz.$$
Thus we get
\begin{equation}\label{TT}T(z)=\pm \Re \frac{\sin \frac{t}{2} \log[\frac{1-z^2}{1+z^2}]-\log\frac{1+z}{1-z} \sqrt{\cos t} \tan \left[\frac{t}{2}\right]}{\pi }. \end{equation}
Then $$T(z)=\pm \frac{\sin \frac{t}{2} \log[\frac{|1-z^2|}{|1+z^2}|]-\log\frac{|1+z|}{|1-z|} \sqrt{\cos t} \tan \left[\frac{t}{2}\right]}{\pi },$$ so $T(z) \to\pm\infty$ when $z\to \pm 1$ or $z\to \pm \imath$. Moreover its noparametric parametrization $(u,v, \mathbf{f}^{t}(u,v))$, $(u,v)\in\mathcal{T}$ satisfies the relation $\mathbf{f}^{t}(u,v))\to \pm \infty$ when $z=(u,v)\to \zeta \in \partial \mathcal{T}$.
An example of a Scherk's type minimal graph is shown in the figure~3.2 below.

\begin{figure}[htp]\label{f2}
\centering
\includegraphics{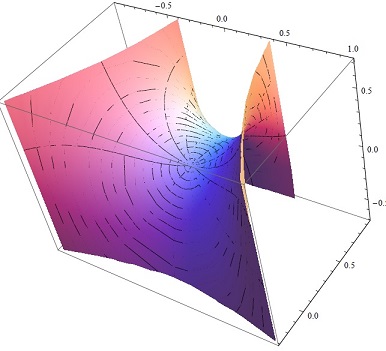}
\caption{A generalized Scherk's surface. Here $t=\pi/2-0.1$}
\end{figure}

Since $S^t$ is symmetric with respect to the plane $xOz$, it follows that it is a graph of a function $\mathbf{f}^t$ defined in the unit disk which is symmetric with respect to the $u-$axis. This implies that $\mathbf{f}^t(-u,v)=\mathbf{f}^t(u,v)$. So $$D_u \mathbf{f}^{t}(-u,v)=-D_u \mathbf{f}^{t}(u,v). $$ and so
 $$D_u \mathbf{f}^t(0,v)=0.$$ Thus \begin{equation}\label{xyzero}D_{uv}\mathbf{f}^t(0,v)=0 \text{ for every $v$.}\end{equation}
 Thus we have proved the following proposition.
 \begin{proposition}\label{defshre}
 For any $t\in(0,\pi/2]$ there is an isosceles trapezoid $$\mathcal{T}^t=\mathcal{T}(e^{\imath\alpha(t)},e^{\imath\beta(t)}, e^{\imath\gamma(t)}, e^{\imath\delta(t)})$$ with the vertices at the unit circle, with bases parallel to the $u-$axis  and a Scherk's type minimal surface $$S^t=\{(u,v, \mathbf{f}^t(u,v)): (u,v)\in \mathcal{T}^t\}$$ so that
 $$\mathbf{f}^t(z)\to \left\{
             \begin{array}{ll}
               +\infty, & \hbox{if $z\to \zeta$ when $\zeta\in (e^{\imath\alpha(t)}, e^{\imath\beta(t)})\cup (e^{\imath\gamma(t)}, e^{\imath\delta(t)})$;} \\
               -\infty, & \hbox{if $z\to \zeta$ when $\zeta\in (e^{\imath\beta(t)}, e^{\imath\gamma(t)})\cup (e^{\imath\delta (t)}, e^{\imath\alpha(t)})$.}
             \end{array}
           \right.$$
 Moreover $D_{uv} \mathbf{f}^t(0,0)=D_u\mathbf{f}^t (0,0)=0$.

Further for $t\in\left(t_\circ,\frac{\pi}{2}\right]$, where $t_\circ = 2\tan^{-1}\sqrt{\frac{1}{2}(\sqrt{5}-1)}$, the trapezoid $\mathcal{T}$ contains zero and the Gaussian curvature of $S_t$ at the point  $\mathbf{w} $ above $0$ is equal to $\mathcal{K}(\mathbf{w})=-\kappa^2(t)$, where $\kappa(t)$ is defined in \eqref{kappa}.
Furthermore, $\kappa^2(t)\le \frac{\pi^2}{2}$ for every $t$ and $\lim_{t\to t_\circ}=0$ and $\kappa^2(t)$ is an increasing diffeomorphism of $(t_\circ, \pi/2]$ onto $(0,\pi^2/2]$.

For $t=\pi/2$ the obtained surface is the standard Scherk's minimal graph surface over the square.
 \end{proposition}

\begin{proof}[Proof of Theorem~\ref{prejprej}]
In order to prove Theorem~\ref{prejprej}, we will derive a useful formula for $\mathbf{f}_{uv}$, of a non-parametric minimal surface $w=\mathbf{f}(u,v)$. Namely we will express $\mathbf{f}_{uv}$ as a function of Enneper-Weisstrass parameters.
Assume that $q(z) = a(z) + \imath b(z)=\sqrt{\omega(z)}$ and $p$ are Enneper-Weisstrass parameters of a minimal disk $S=\{(u(z), v(z), T(z)), z\in \D\}=\{(u,v,\mathbf{f}(u,v)): (u,v)\in \D\}$ over the unit disk. Here $f=u+iv$ and $\bar f_z=\omega(z) f_z$.

Then the unit normal at $\mathbf{w}\in S$, in view of \cite[p.~169]{Duren2004}  is given by $$\mathbf{n}_{\mathbf{w}}=-\frac{1}{1+|q(z)|^2}(2\Im q(z), 2\Re q(z), -1+|q(z)|^2).$$ It is also given by the formula
$$\mathbf{n}_{\mathbf{w}}=\frac{1}{\sqrt{1+\mathbf{f}_u^2+\mathbf{f}_v^2}}\left(-\mathbf{f}_u,-\mathbf{f}_v,1\right).$$  Then we have the relations
 \begin{equation}\label{firsti}\mathbf{f}_v (u(x,y),v(x,y))=\frac{2 a(x,y)}{-1+a(x,y)^2+b(x,y)^2}
  \end{equation}
  \begin{equation}\label{secondi}\mathbf{f}_u(u(x,y),v(x,y))=\frac{2 b(x,y)}{-1+a(x,y)^2+b(x,y)^2}.\end{equation}
By differentiating \eqref{firsti} and \eqref{secondi}  w.r.t. $x$  we obtain the equations
 \begin{equation}\label{firsti1}v_x \mathbf{f}_{uv}(u,v)+u_x \mathbf{f}_{uu}(u,v)=-\frac{4 a b a_x+2 \left(1-a^2+b^2\right) b_x}{\left(-1+a^2+b^2\right)^2}\end{equation}
  \begin{equation}\label{secondi1}v_x\mathbf{f}_{vv}(u,v) +u_x \mathbf{f}_{uv}(u,v)=-\frac{4 a b b_x+2 \left(1-a^2+b^2\right) a_x}{\left(-1+a^2+b^2\right)^2}.\end{equation}
Now recall the minimal surface equation
 \begin{equation}\label{mse}\left(1+\mathbf{f}^2_u(u,v)^2\right)\mathbf{f}_{vv}(u,v)+\left(1+\mathbf{f}^2_v(u,v)^2\right) \mathbf{f}_{uu}(u,v)=2 \mathbf{f}_v(u,v) \mathbf{f}_u(u,v)  \mathbf{f}_{uv}(u,v) \end{equation}
 From \eqref{firsti}, \eqref{secondi}, \eqref{firsti1},\eqref{secondi1} and \eqref{mse} we get
$$\mathbf{f}_{uv}=\frac{M}{N}$$ where
\[\begin{split}M&=-2 (a^4+2 a^2 (-1+b^2)+(1+b^2)^2) ((1+a^2-b^2) a_x+2 a b b_x) u_x\\&-2((1+a^2)^2+2 (-1+a^2) b^2+b^4) (2 a b a_x+(1-a^2+b^2) b_x) v_x\end{split}\] and
\[\begin{split}N&=(1-a^2-b^2)^2 \\&\times ((a^4-2 a^2 (1-b^2)+(1+b^2)^2) u_x^2+8 a b u_x v_x+((1+a^2)^2-2 (1-a^2) b^2+b^4) v_x^2).\end{split}\]
Let $q(z)=a+\imath b=r e^{it}$, $q'(z) = a_x+\imath b_x=R e^{is}$ and $p=Pe^{im}$.  Because $u_x=\Re (p(1+q^2))$, and $v_x=-\Re (\imath p(1-q^2))$,
after straightforward calculation we get
$$\mathbf{f}_{uv}=-\frac{2 R \left(\cos[m-s]-r^4 \cos[m-s+4 t]\right)}{P \left(1-r^2\right)^3 \left(1+r^2\right)}$$
which can be written as
\begin{equation}\label{fexpli}\mathbf{f}_{uv}=-\frac{2\Re \left[p(1-q^4)\overline{q'}\right]}{|p|^2(1-|q|^2)^3(1+|q|^2)}.\end{equation}

Now we continue to prove Theorem~\ref{prejprej}.
The solution of \eqref{beleq} with such initial conditions exists and is unique \cite[Theorem~A\&~Theorem~1]{zbMATH05159460} and maps the unit disk onto a quadrilateral $Q(a_0,a_1,a_2,a_3)$ whose vertices $a_0,a_1,a_2,a_3$, $a_4=a_0$ belongs to the unit circle. Moreover by  \cite[Theorem~B]{zbMATH05159460}, there are four points $b_k=e^{\imath\alpha_k}, \ k=0,1,2,3$, $b_4=b_0, b_5=b_1$, $$F(e^{it})=\sum_{k=1}^4 a_k I_{(\alpha_k, \alpha_{k+1})}(t).$$ Here $F$ is the boundary function of $f$. Therefore (\cite[p.~63]{Duren2004}) we can conclude that $$f_z(z) = \sum_{k=1}^4 \frac{d_k}{z-b_k},$$  and that  $$\bar{f}_z(z) = -\sum_{k=1}^4 \frac{\overline{d_k}}{z-b_k},$$ where
$$d_k = \frac{a_k -a_{k+1}}{2\pi \imath}.$$
 Therefore the third coordinate of conformal parameterisation is
 $$T(z) =\pm 2\Re \imath \int_0^z \sqrt{f_z\bar f_z}dz$$ thus when $z$ is close to $b_k$, then $$T(z) =\pm|d_k|^2\log|1-z/b_k|+O(z-b_k).$$
Thus when $z\to b_k$, $T(z)\to \pm \infty$. This implies that $\mathbf{f}(z)\to \pm\infty$ if $z\to a\in(a_k, a_{k+1})$.
Since $$q(z) =\frac{w+\frac{\imath \left(1-w^4\right) z}{\left|1-w^4\right|}}{1+\frac{\imath\overline{w} \left(1-w^4\right) z}{\left|1-w^4\right|}},$$ we get $$q(0) = w \ \ \text{and}\ \ q'(0)=\frac{\imath \left(1-w^4\right) \left(1-|w|^2\right) }{\left|1-w^4\right|}.$$

Now \eqref{finoser} follow from \eqref{eq:curvatureformula}.

Further
$$p(0)(1-q(0)^4)\overline{q'(0)}=-\imath f_z(0)|1-w^4|(1-|w|^2).$$
So in view of the formula \eqref{fexpli} we conclude $\mathbf{f}^\diamond_{uv}=0.$

Now we assert that \begin{equation}\label{needed}|\mathcal{K}_{S}(\mathbf{w})|< |\mathcal{K}_{S^{\diamond}}(\mathbf{w})|,\end{equation}
  and what is the same \begin{equation}\label{needed1}W_S^2|\mathcal{K}_{S}(\mathbf{w})|< W_{S^\diamond}^2|\mathcal{K}_{S^{\diamond}}(\mathbf{w})|.\end{equation}
  Assume the converse $|\mathcal{K}_{S}(\mathbf{w})|\ge  |\mathcal{K}_{S^{\diamond}}(\mathbf{w})|$ and argue by a contradiction.  Then as in \cite{FinnOsserman1964}, by using the dilatation $L(\zeta) = \lambda \zeta$ for some $\lambda\ge 1$ we get the surface
$$S_1=L(S)=\{(u,v,\lambda \mathbf{f}\left(\frac{u}{\lambda}, \frac{v}{\lambda}\right): |u+\imath v|<{\lambda}\},$$ whose Gaussian curvature
$$\mathcal{K}_1(\mathbf{w})=\frac{\frac{1}{\lambda^2} \left(\mathbf{f}_{uu}(0,0)\mathbf{f}_{vv}(0,0)-\mathbf{f}_{uv}(0,0)^2\right)}{(1+\mathbf{f}_u(0,0)^2+\mathbf{f}_v(0,0)^2)^2}.$$ Observe that such transformation does not change the unit normal at $\mathbf{w}$.

Then there is $\lambda_\ast\ge 1$ so that $\mathcal{K}_1(\mathbf{w})=\mathcal{K}_{S^\diamond}(\mathbf{w})$. Let $$\mathbf{f}^\ast (u,v) =\lambda_\ast \mathbf{f}\left(\frac{u}{\lambda_\ast}, \frac{v}{\lambda_\ast}\right).$$
From
$\mathbf{n}_\diamond=\mathbf{n}_\ast$ we get
\begin{equation}\label{nowafter}\mathbf{f}^\diamond_{u}(0,0)=\mathbf{f}^\ast_{u}(0,0),  \ \mathbf{f}^\diamond_{v}(0,0)=\mathbf{f}^\ast_{v}(0,0).\  \end{equation}

  Further we have $$(1+(\mathbf{f}^\ast_{u}(0,0))^2)\mathbf{f}^\ast_{vv} (0,0)-2 \mathbf{f}^\ast_{u}(0,0)\mathbf{f}^\ast_{v}(0,0)\mathbf{f}^\ast_{uv} (0,0)+(1+(\mathbf{f}^\ast_{v}(0,0))^2)\mathbf{f}^\ast_{uu} (0,0)=0,$$
$$(1+(\mathbf{f}^\diamond_{u}(0,0))^2)\mathbf{f}^\diamond_{vv} (0,0)-2 \mathbf{f}^\diamond_{u}(0,0)\mathbf{f}^\diamond_{v}(0,0)\mathbf{f}^\diamond_{uv} (0,0)+(1+(\mathbf{f}^\diamond_{v}(0,0))^2)\mathbf{f}^\diamond_{uu} (0,0)=0,$$  $$ \mathbf{f}^\diamond_{uv}(0,0)=\mathbf{f}^\ast_{uv}(0,0)$$ and the equation
$$\frac{ \left(\mathbf{f}^\ast _{uu}(0,0)\mathbf{f}^\ast_{vv}(0,0)-\mathbf{f}^\ast_{uv}(0,0)^2\right)}{(1+\mathbf{f}^\ast_u(0,0)^2+\mathbf{f}^\ast_v(0,0)^2)^2}=
\frac{ \left(\mathbf{f}^\diamond _{uu}(0,0)\mathbf{f}^\diamond_{vv}(0,0)-\mathbf{f}^\diamond_{uv}(0,0)^2\right)}{(1+\mathbf{f}^\diamond_u(0,0)^2+\mathbf{f}^\diamond_v(0,0)^2)^2}.$$
We can also w.l.g. assume that $\mathbf{f}^\ast_{uu} $ and $\mathbf{f}^\diamond_{uu} $ as well as $\mathbf{f}^\ast_{vv} $ and $\mathbf{f}^\diamond_{vv} $ have the same sign. If not, then we choose $\lambda_\ast\le -1$ and repeat the previous procedure with $$S_1=L(S)=\{(u,v,\lambda \mathbf{f}\left(\frac{u}{\lambda}, \frac{v}{\lambda}\right): |u+\imath v|<{|\lambda|}\}.$$
Thus the function $F(u,v) = \mathbf{f}^\ast(u,v)- \mathbf{f}^\diamond(u,v)$ has all derivatives up to the order $2$ equal to zero in the point $w=0$.

To continue the proof we use the following lemma
\begin{lemma}\label{leci}
Assume that the quadrilateral $Q=Q(a,b,c,d)$ is inscribed in the unit disk, and assume that $\zeta=\mathbf{f}(u,v)$ is a Scherk's type minimal surface $S$ above $Q$. i.e. assume that $\mathbf{f}(u,v)\to +\infty$ when $\zeta=u+iv \to w\in  (a,b)\cup (c,d)$ and $\mathbf{f}(u,v)\to -\infty$ when $\zeta=u+iv \to w\in  (b,c)\cup (a,d)$. Then there is not any other bounded minimal graph $\zeta=\mathbf{f}_1(u,v)$ over a domain $\Omega$ that contains  $Q$ which has the same Gaussian curvature, the same Gaussian normal, and the same mixed derivative  at the same point $\mathbf{w}\in Q$ as the given surface $S$.
\end{lemma}
\begin{proof}[Proof of Lemma~\ref{leci}]
We observe that \cite[Proof of Proposition~1]{FinnOsserman1964} works for every Scherk's type minimal surface, so if we would have a  bounded minimal surface having the all derivatives ap to the order 2 equal to zero, then such non-parametric parameterizations $\mathbf{f}$ and $\mathbf{f}_1$, in view of \cite[Lemma~1]{FinnOsserman1964} will satisfy the relation $F(z)=\mathbf{f}(z)-\mathbf{f}_1(z)=O(\zeta^N(z))$, $N\ge 3$, where $\zeta$ is a certain homeomorphism between two open sets containing $0$. Then by following the proof of \cite[Proof of Proposition~1]{FinnOsserman1964} (second part) we get that this is not possible, because Sherk's type surface has four "sides" but the number $2N$ is bigger or equal to $6$ which is not possible.
\end{proof}

This leads to the contradiction so \eqref{needed} is true. To finish the  proof of Theorem~\ref{prejprej} we need to prove the sharpness. It is similar to the proof of sharpness of Theorem~\ref{th:theor2} below so we omit it.
\end{proof}
\begin{proof}[Proof of Theorem~\ref{th:theor}]
 Assume that $S=\{(u,v,\mathbf{f}(u,v)): (u,v)\in \D\}$ is any surface above the unit disk and assume that $\mathbf{f}(0,0)=0$. Assume also that we have rotated the unit disk so that $\mathbf{f}_{uv}(0,0)=0$ and $\mathbf{f}_u(0,0)=0$. Namely if $h=e^{ic}$ and $\mathbf{f}^c(z) = \mathbf{f}(e^{ic}z)$. Then as in Example~\ref{forward} $$\mathbf{f}^c_u(0,0)=\nabla_h \mathbf{f}(0,0)=0.$$ Further
$$\mathbf{f}^c_{uv}(0,0)=\nabla^2_{h,\imath h} \mathbf{f}(0,0)=0.$$
 Let $\mathbf{f}_v(0)=V$ and assume w.l.g that $V>0$.
Then the Gaussian normal of $\mathbf{w}\in S$ is \begin{equation}\label{secondn}\mathbf{n}=\frac{1}{\sqrt{1+V^2}}(0,-V, 1).\end{equation}

 The Gauss map of $S_t$ above $0=f(z_\circ)$ can be expressed as (see \cite[p.~169]{Duren2004})  $$\mathbf{N}_t=-\frac{1}{1+|a(t)|^2}(2\Im a(t), 2\Re a(t), -1+|a(t)|^2),$$ where
$a(t) = q(z_\circ)$. By \eqref{qqq} and \eqref{zezero} we have
  \begin{equation}\label{inve}a(t)=\frac{\sqrt{\cos t}+\left(\cos \left[\frac{t}{2}\right]+\sin \left[\frac{t}{2}\right]\right) \tan \left[\frac{1}{8} \pi  \sqrt{\cos t} \csc\left[\frac{t}{2}\right]^3 \sin t\right]}{\cos \left[\frac{t}{2}\right]+\sin \left[\frac{t}{2}\right]+\sqrt{\cos t} \tan \left[\frac{1}{8} \pi  \sqrt{\cos t} \csc\left[\frac{t}{2}\right]^3 \sin t\right]} .\end{equation}
We need to find $t $ so $\mathbf{N}_t=\mathbf{n}$, where $\mathbf{n}$ is the unit normal at the second minimal surface above $0$ defined in \eqref{secondn}.

 Since the function $a(t)$ is continuous for $t\in[t_\circ,\pi/2]$ and $a(\pi/2)=0$ and $$a(t_\circ) = a\left(2 \tan^{-1}\left[\sqrt{\frac{1}{2} \left(-1+\sqrt{5}\right)}\right]\right)=1,$$
there is $t_0\in(t_\circ , \pi/2)$ so that $$a(t_0)=\frac{-1+\sqrt{1+V^2}}{V}.$$
In this case  $\mathbf{N}_{t_0}=\mathbf{n}$.

Assume now that $$S^{\diamond}=\{((u,v), \mathbf{f}^\diamond(u,v)): (u,v)\in \D\},$$ is the Scherk 's type surface above the trapezoid  $\mathcal{T}=\mathcal{T}^{t_0}$ so that $\mathbf{f}^\diamond(0,0)=\mathbf{f}(0,0)=0$.

Let $\mathbf{w}=(0,0,0)$. Then instead of \eqref{nowafter} we have \begin{equation}\label{nowafter}\mathbf{f}^\diamond_{u}(0,0)=\mathbf{f}^\ast_{u}(0,0)=0,  \ \mathbf{f}^\diamond_{v}(0,0)=\mathbf{f}^\ast_{v}(0,0)=V.\  \end{equation} Then as in the proof of Theorem~\ref{prejprej} we obtain  that \begin{equation}|\mathcal{K}_{S}(\mathbf{w})< |\mathcal{K}_{S^{\diamond}}(\mathbf{w})|,\end{equation} and

\begin{equation}W_{S}^2|\mathcal{K}_{S}(\mathbf{w})< W_{\diamond}^2|\mathcal{K}_{S^{\diamond}}(\mathbf{w})|.\end{equation}

 Lemma~\ref{leci} works also for the trapezoid instead of the square. The only important thing is that the mapping $T(z)$ defined in \eqref{TT} tends to $\pm \infty$ as $z\to \pm 1$ of $z\to  \pm \imath$. This implies that $\mathbf{f}^\diamond(z)\to \pm\infty$ if $z\to \zeta$, where $\zeta$ belongs to an open side of the trapezoid.
 Now subsections~\ref{subsub1} and \ref{subsub2} imply that $$|\mathcal{K}_{S}(\mathbf{w})\le W_{S}^2|\mathcal{K}_{S}(\mathbf{w})\le \pi^2/2$$ what we needed to prove.
\end{proof}

\begin{proof}[Proof of Theorem~\ref{th:theor2}]
Assume that $S^t$ is as in Proposition~\ref{defshre}. Since $\kappa:[t_\circ, \pi/2]\to [0,\pi/\sqrt{2}]$ is increasing (see subsection~\ref{subsub1}), the function $a(t)=|q(z_\circ)|: [t_\circ, \pi/2]\to [0,1]$ is decreasing. Further the angle $\theta=\arccos\frac{1-|q(0)|^2}{1+|q(0)|^2}$ of the unit normal is uniquely determined by $|q(z_\circ)|$. It follows that there is a bijective correspondence between the curvature at $\mathbf{w}\in S^t$ and the angle that tangent plane $TS^t_\mathbf{w}$ forms with the $v-$axis. In this way it is determined a continuous decreasing function $\Psi(\theta)=|\mathcal{K}(\mathbf{w})|:[0,\pi/2]\to [0,\pi^2/2]$.
The proof of the first part is the same as the proof of Theorem~\ref{th:theor}.

Prove the second part. A similar statement for the case that the tangent plane is horizontal has been proved in \cite[Proposition~3]{FinnOsserman1964}. However that proof does not work in this case.
Assume that  $\omega= (q(z))^2$ where $q$ is defined in \eqref{qqq}.  Also assume that $t\in(t_\circ, \pi/2]$. Let $f$ be as in \eqref{after}. Then $f$ is a solution of  Beltrami equation $\overline{f}_z=\omega f_z$ satisfying the initial conditions $$f_z(0)=p=\frac{\imath \sec\left[\frac{t}{2}\right] (-1+\cos t-\sin t)}{\pi }$$ (in view of \eqref{pmtp} and \eqref{pptp}) and $f(0)=\imath \sqrt{\cos t}$ (because of \eqref{be}). Further $f$ maps the unit disk onto the convex trapezoid $\mathcal{T}$.  This implies that $\tilde f=\imath f$ maps the unit disk onto the trapezoid $\imath\mathcal{T}$ and satisfies the equation $\overline{ f}_z=-\omega f_z$ with the initial condition $\tilde f(0)=- \sqrt{\cos t}$ and $\tilde f_z(0)>0$. Recall also that $f(z_\circ)=0$, where $z_\circ$ is defined in \eqref{zezero}.

Further, for $0<k<1$ assume that  $\omega_k=k^2 e^{-\imath\pi/2}\omega.$  Then solve the second Beltrami equation $\overline{f}_z=\omega_k f_z$ that map the unit disk $\D$ onto itself  satisfying the initial condition $f(0) = -\sqrt{\cos t}$ and $f_z(0)>0$ \cite{HengartnerSchober1986}. This mapping exists and is unique \cite[p.~134]{Duren2004}. Then this mapping produces a minimal surface $S_k^t$ over the unit disk. Moreover for $k=n/(n+1),$ the sequence $f_n$ converges (up to some subsequence) in compacts of the unit disk, to a mapping $f^\circ$ that maps the unit disk into the unit disk. By using again the uniqueness theorems  \cite[Theorem~B\&~Theorem~1]{zbMATH05159460}, because $f^\circ(0)=\tilde f(0)=-\sqrt{\cos t}$ and $f^\circ_z(0)>0$, it follows that $f^\circ \equiv \tilde f$. Let $\mathbf{w}_n$ be the point above $0$ of minimal surface $S^t_n$. Let $z_n\in \D$, so that $f_n(z_n)=0$. Then $\mathbf{w}_n$ converges to $\mathbf{w}$. Moreover the Gaussian curvatures $\mathcal{K}_n(\mathbf{w}_n)$ of $S_n^t$, in view of the formula \eqref{eq:curvatureformula}, is equal to $$- \frac{4|q_n'(z_n)|^2}{|p_n(z_n)|^2(1 + |q_n(z_n)|^2)^4}$$ and converges to the Gaussian curvature $\mathcal{K}(\mathbf{w})=-\kappa^2(t).$ Namely $z_n=f_n^{-1}(0)$, and therefore $\lim_{n\to \infty}z_n=\lim_{n\to\infty}f_n^{-1}(0)=f^{-1}(0)=z_\circ$, because $f_n^{-1}$ and also $f^{-1}$ are quasiconformal in a disk around $0$ and the family is normal. Also $q_n$ and $p_n$ and $q_n'$ converges in compacts to the corresponding $q$, $p$ and $q'$. We proved that for a fixed $\theta$ the inequality \eqref{eq:FinnOsserman2} cannot be improved. In a similar way we can prove the rest of the theorem.
\end{proof}
\begin{remark}\label{remica}
It follows from Section~\ref{sectio2}, see \eqref{afterv}, that the mapping $f$ satisfies the conditions $f(z_\circ)=0$ and $\imath f_z(z_\circ)>0$. So the mapping $\hat f$ defined by $\hat f(z) =\imath f\left(\frac{z+z_\circ}{1+zz_0}\right)$ satisfies the conditions $\hat f(0)=0$ and $\hat f_z(0)>0$. Moreover it satisfies the Beltrami equation \eqref{beleq} with $w=\imath a(t)$, where $a(t)$ is defined in \eqref{inve}. In this case the given trapezoid is symmetric w.r.t real axis.
\end{remark}
Now \eqref{finoser} and Theorem~\ref{th:theor} (or the result of Finn and Osserman), implies the following lemma.
\begin{lemma}
Assume that $f$  solves the equation $$\bar f_z(z) = z^2 f_z(z),$$ with the initial conditions $f(0)=0$ and $f_z(0)>0$. Assume also that $f$ is a limit of harmonic diffeomorphisms $f_n:\D\to D_n \supseteq \D$, whose second dilatations are squares of holomorphic functions, with initial conditions $f_n(0)=(\bar{f}_{n})_z(0)=0$. Then the sharp inequality $$|f_z(0)|\ge \frac{2\sqrt{2}}{\pi}$$ holds.
\end{lemma}
\begin{proof}
The only important think is that $f_n$ can be lifted to a minimal surface, whose projection contains the unit disk with $f_n(0)=(\bar{f}_{n})_z(0)=0$, so the result follows from \eqref{finoser} and the result of Finn and Osserman (or Theorem~\ref{th:theor}).
\end{proof}
To prove Corollary~\ref{coro} we also need the following lemma.
\begin{lemma}
Assume that $f$ is a limit of harmonic diffeomorphisms $f_n$ of the unit disk onto $D_n \supseteq \D$ with squared second holomorphic dilatations, that solve the equation $$\bar f_z(z) = \left(\frac{w+e^{is}z}{1+e^{is}\overline{w}z}\right)^2 f_z(z),$$ with the initial conditions $f(0)=0$ and $f_z(0)>0$. Then we have the inequality
\begin{equation}\label{improv}|f_z(0)|\ge \frac{2\sqrt{2}}{\pi}\frac{(1-|f(-we^{-is})|)}{1-|w|^2}.\end{equation}
\end{lemma}
\begin{proof}
Let $$f^1(z) = \frac{1}{1-|f(-we^{-is})|}\left(f\left(\frac{e^{-\imath s} (w-z)}{-1+z \overline{w}}\right)-f(-we^{-is})\right).$$
Then $f^1$ solves the Beltrami equation $$\overline{f}^1_z(z)=z^2 f^1_z(z)$$ and $f^1(0)=0$, $f^1_{\bar z}(0)>0$. Let $f_n$ be a mapping defined by

$$f_n^1(z) = \frac{1}{1-|f_n(-we^{-is})|}\left(f_n\left(\frac{e^{-\imath s} (w-z)}{-1+z \overline{w}}\right)-f_n(-we^{-is})\right).$$

Then the second dilatation of $f^1_n$ is the square of an analytic function and it satisfies the initial conditions $f^1_n(0)=(\bar{f^1}_{n})_z(0)=0$.
 Therefore by Lemma~\ref{leci}, in view of  \eqref{finoser} we get  $|f^1_z(0)|\ge \frac{\pi}{2\sqrt{2}}$,  and this implies the claimed inequality.
\end{proof}
\begin{proof}[Proof of Corollary~\ref{coro}]
Let $S: \zeta=\mathbf{f}(u,v)$ be a non-parametric minimal surface over the unit disk and assume that $$\mathbf{n}_{\mathbf{w}}=-\frac{1}{1+|w|^2}(2\Im w, 2\Re w, -1+|w|^2),$$ is its Gaussian normal at the \emph{centre}. Let $f=f_w$ be the solution that is provided to us by Theorem~\ref{prejprej} that produces the Scherk type minimal surface $S^\diamond$. Let $\mathcal{K}=\mathcal{K}_{S^\diamond}(\mathbf{w})$. In view of Theorem~\ref{prejprej}, we only need to estimate the curvature $\mathcal{K}$.  Now we have the estimate
\begin{equation}\label{eq:weakestK1}
	\begin{split}|\mathcal{K}| & =\frac{4(1-|w|^2)}{|f_z(0)|^2(1+|w|^2)^4}\\& \le
	\frac{16\pi^2}{27} \frac{\bigl(1-|w|^2\bigr)^2 \bigl(1+|w|^4\bigr)}{(1+|w|^2)^4}.\end{split}
\end{equation}
Write $|w|=r$ and consider the function
\begin{equation}\label{eq:h}
	G(r)= \frac{16\pi^2}{27}\frac{\bigl(1-r^2\bigr)^2 \bigl(1+r^4\bigr)}{(1+r^2)^4}, \quad 0\le r\le 1.
\end{equation}
Note that \eqref{eq:weakestK1} can be written in the form
\begin{equation}\label{eq:estK2}	
	|\mathcal{K}|  \le  G(r).
\end{equation}
Now the proof of Theorem~\ref{prejprej} implies that, $f$ is a limit of a sequence $f_n$ satisfying Corollary~\ref{coro}. In view of \eqref{finoser} and \eqref{improv} and harmonic Schwarz lemma: $|f(w)|\le \frac{4}{\pi}\tan^{-1}(|w|)$,  we get
\[\begin{split}|\mathcal{K}(\mathbf{w})|&=\frac{4 \left(1-|w|^2\right)^2}{\left(1+|w|^2\right)^4 |f_z(0)|^2}
\\& \le \frac{4 \left(1-|w|^2\right)^2}{\left(1+|w|^2\right)^4 |\frac{2\sqrt{2}}{\pi}\frac{(1-|f(-w)|)}{1-|w|^2}|^2}
\\& \le\frac{\pi^2}{2} \frac{ \left(1-|w|^2\right)^4}{\left(1+|w|^2\right)^4 (1-|f(-w)|)^2}
\\& \le \frac{\pi^2}{2} \frac{ \left(1-|w|^2\right)^4}{\left(1+|w|^2\right)^4 (1-\frac{4}{\pi}\tan^{-1}(|w|))^2} :=H(|w|).
\end{split}\]
From the previous relations and \eqref{eq:estK2} we conclude that $$\mathcal{K}(\mathbf{w})\le \max_{r\in[0,1]}\min\{G(r), H(r)\}.$$
Let $r_\diamond \approx   0.067344733$ be the solution of the equation $$ G(r)=H(r), \ \ r\in(0,1),$$ where $G$ is defined in \eqref{eq:h}. It can be easily proved that $H$ increases  in $r\in(0,r_\diamond)$ and $G$ decreases in $(0,1)$. Therefore, $\mathcal{K}(\mathbf{w})<G(r_\diamond) \approx 5.6918$.

Further, by using \eqref{firsti} and \eqref{secondi} we get  $$W^2=1+\mathbf{f}^2_u+\mathbf{f}^2_v=\frac{\left(1+|w|^2\right)^2}{\left(1-|w|^2\right)^2},$$  where $w=|q(0)|$. Therefore we get that
$$\mathcal{K}(\mathbf{w})W^2 \le \frac{16 \pi ^2 \left(1+r_\diamond^4\right)}{27 \left(1+r_\diamond^2\right)^2}\approx 5.79608.$$
\end{proof}
\subsection*{Acknowledgements}
The author is partially supported by a research fund of
University of Montenegro. I wish to thank Franc Forstneri\v c for fruitful conversation and his remarks which led
to improved presentation and I also thank Antonio Ros for encouraging me to work in this problem.




{\bibliographystyle{abbrv} \bibliography{references}}






\noindent David Kalaj

\noindent University of Montenegro, Faculty of Natural Sciences and Mathematics, 81000, Podgorica, Montenegro

\noindent e-mail: {\tt davidk@ucg.ac.me}

\end{document}